\documentclass[10pt]{amsart}

\usepackage{amsopn}
\usepackage{amsmath,amsthm,amssymb, amscd, amsfonts}

\usepackage[pdftex]{hyperref}  

\theoremstyle{theorem}
\newtheorem{teo}{\bf Theorem}[section]
\newtheorem{cor}[teo]{\bf Corollary}

\newtheorem{prop}[teo]{\bf Proposition}
\newtheorem{lema}[teo]{\bf Lemma}

\theoremstyle{definition}

\theoremstyle{remark}
\newtheorem{nota}[teo]{\it Remark}

\numberwithin{equation}{section}

\newcommand\Ad{\operatorname{Ad}}

\newcommand\Aff{\operatorname{Aff}}

\newcommand\Exp{\operatorname{Exp}}

\newcommand\Iso{\operatorname{Iso}}
\newcommand\Lie{\operatorname{Lie}}

\newcommand\Spin{\operatorname{Spin}}

\newcommand\Tr{\operatorname{Tr}}

\newcommand\aff{\mathfrak{aff}}
\renewcommand\gg{\mathfrak{g}}
\newcommand\ga{\mathfrak{a}}

\newcommand\gh{\mathfrak{h}}
\newcommand\gk{\mathfrak{k}}
\newcommand\gm{\mathfrak{m}}
\newcommand\tr{\mathfrak{tr}}

\newcommand\ch{\mathcal{H}}

\newcommand\cv{\mathcal{V}}

\makeatletter\renewcommand\theenumi{\@roman\c@enumi}\makeatother

\begin{document}

\title{On the affine group of a normal homogeneous manifold}

\author{Silvio Reggiani} \address{Facultad de Matem\'atica, Astronom\'\i a y F\'\i sica, Universidad Nacional de
C\'ordoba, Ciudad Universitaria, 5000 C\'ordoba, Argentina}
\email{reggiani@famaf.unc.edu.ar}

\date{\today}

\thanks {2000 {\it  Mathematics Subject Classification}.  Primary 53C30; Secondary 53C35}

\thanks {{\it Key words and phrases}. Naturally reductive, normal homogeneous, canonical connection, transvection group,
affine group, isometry group}

\thanks{Supported by CONICET. Partially supported by ANCyT, Secyt-UNC and CIEM}

\begin{abstract}
A very important class of homogeneous Riemannian manifolds are the
so-called normal homogeneous spaces, which have associated a
canonical connection. In this work we obtain geometrically the
(connected component of the) group of affine transformations with
respect to the canonical connection for a normal homogeneous space.
The na\-tu\-ra\-lly reductive case is also treated. This completes
the geometric calculation of the isometry group of naturally
reductive spaces. In addition, we prove that for normal homogeneous
spaces the set of fixed points of the full isotropy is a torus. As
an application of our results it follows that the holonomy group of
a homogeneous fibration is contained in the group of (canonically)
affine transformations of the fibers, in particular this holonomy
group is a Lie group (this is a result of Guijarro and Walschap).
\end{abstract}

\maketitle

\section{Introduction}\label{introduction}

Compact normal homogeneous Riemannian manifolds, or, more
ge\-ne\-ra\-lly, naturally reductive spaces, are a very important
family of homogeneous spaces. These spaces appear in a natural way
as a generalization of symme\-tric spaces. Associated to a
naturally reductive space there is a canonical connection. For
symmetric spaces the Levi-Civita connection is a canonical
connection. For general naturally reductive spaces this is not
anymore true, but they have the following important property that
characterizes them: there exists a canonical connection which has
the same geodesics as the Levi-Civita connection.

Examples of naturally reductive and normal homogeneous spaces are
the so-called isotropy irreducible spaces, strongly or not, which
were classified by Wolf \cite{Wo} (1968), in the strong case, and
Wang and Ziller \cite{WZ2} (1991) in the general case.

If  $M = G/H$ is a compact naturally reductive space with associated
canonical connection $\nabla^c$, then the group $\Aff(\nabla^c)$, of
$\nabla^c$-affine transformations, is a compact subgroup of
$\Iso(M)$ (this is done in Section \ref{G-invariant}). Moreover, if
$M$ is a normal homogeneous space, then we shall prove that the flow
of any $G$-invariant field is $\nabla^c$-affine and therefore, an
isometry (see Theorem \ref{2.4} and Corollary \ref{2.5}). The aim of
this paper is to compute, in a geometric way, the (connected
component of the identity of the) full $\nabla^c$-affine group.
Namely, our main result is the following theorem.

\begin{teo}\label{main}
Let $M = G/H$ be a compact normal homogeneous space, $\gg = \Lie(G)$
and $\gh = \Lie(H)$. Let $\nabla^c$ be the canonical connection
associated to the reductive decomposition $\gg = \gh \oplus
\gh^\bot$. Then, the affine group of the canonical connection is
given by
$$\Aff_0(\nabla^c) = G_1 \times K \quad \text{(almost direct product),}$$
where $K$ denotes the connected subgroup of $\nabla^c$-affine
transformations whose Lie algebra consists of the $G$-invariant
fields and $G_1$ is the semisimple part of $G$.
\end{teo}

Note that given $p \in M$, $K$ is naturally identified with $F_p$,
the connected component by $p$ of the set of fixed points of the
isotropy group $H = G_p$. There is another natural identification
for $F_p$. Let $\bar K(p) = \{k|_{F_p}: k \in G,\, k \cdot F_p =
F_p\}$, then $\bar K(p)$ acts simply transitively on $F_p$. In this
way $F_p$ is naturally identified with a Lie group. Both groups $K$
and $\bar K(p)$ are isomorphic, and their actions on $F_p$ can be
regarded as right and left multiplication respectively.

Theorem \ref{main} also holds for $M$ naturally reductive if one
replaces $\Lie(K)$ by the $\Tr(\nabla^c)$-invariant fields whose
associated flow are $\nabla^c$-affine (see Remark \ref{nrcase}).
Recall that $\Tr(\nabla^c)$ is the group of transvections of the
canonical connection, which is a normal subgroup of
$\Aff(\nabla^c)$.

It is a non-trivial problem to decide, for naturally reductive
metrics, when the group  $\Aff_0(\nabla^c)$ coincides with the full
(connected) isometry group $\Iso_0(M)$. In \cite {OR} it is proved
that this is always true, except for spheres. Such a result depends
strongly on a Berger-type result, the so-called Skew-torsion
Holonomy Theorem, see also \cite{AF} and \cite{Na}. Namely,

\begin{teo} [\cite{OR}] \label{fullisometry}
Let $M = G/H$ be a compact naturally reductive space, which is
locally irreducible, and let $\nabla^c$ be the associated canonical
connection. Assume that $M$ is neither (globally) isometric to a
sphere, nor to a real projective space. Then
\begin{enumerate}
\item $\Iso_0(M) = \Aff_0(\nabla^c).$
\item If $\Iso (M) \not\subset \Aff (\nabla^c)$ then $M$
is isometric to a simple  Lie group, endowed with a bi-invariant
metric.
\end{enumerate}
\end{teo}

We have the following corollary which determines, in a geometric
way, the full isometry group.

\begin{cor}\label{main2}
Let $M = G/G_p$, $p \in M$, be a compact normal homogeneous space. Assume that $M$ is locally irreducible and
that $M \neq S^n$, $M \neq \mathbb RP^n$. Write $G = G_0 \times G_1$ as a almost direct product, where $G_0$ is
abelian and $G_1$ is a semisimple Lie group of the compact type. Then
$$\Iso_0(M) = G_1 \times K,$$
where $K$ denotes the connected component by $p$ of the set of fixed
points of $G_p$ (regarded as a Lie group). In particular, $\Iso(M)$
is semisimple if and only if $K$ is semisimple.
\end{cor}

\smallbreak

We also study the set of fixed points of the isotropy of the full
$\nabla^c$-affine group. By making use of the above corollary we
obtain the following result.

\begin{teo}\label{torus}
Let $M$ be a compact normal homogeneous space and let $S_p$ the connected component
by $p$ of the set of fixed points of $\Iso_0(M)_p$. Then $S_p$ is a torus
(eventually trivial).

\end{teo}

As an application of our main results it follows that the holonomy
group of a homogeneous fibration is contained in the group of
(canonically) affine transformations of the fibers, in particular
this holonomy group must be a Lie group (this is a result of
Guijarro and Walschap \cite{GW}).

 \medbreak

The author would like to thank Carlos Olmos for useful discussions
and comments on the topics of this paper.

\section{Preliminaries}\label{preliminaries}

Let $M = G/H$ be a compact homogeneous Riemannian manifold, where
$G$ is a Lie subgroup of $\Iso(M)$. Let
$$\gg  = \gh \oplus \gm$$
be a reductive decomposition of the Lie algebra of $G$ (where $\gh$
is the Lie algebra of $H$ and $\gm$ is an $\Ad(H)$-invariant
subspace of $\gg$). Associated to this decomposition there is a
canonical connection $\nabla^c$ in $M$, which is $G$-invariant. The
$\nabla^c$-geodesics through $p = eH$ are given by
$$\Exp(tX) \cdot p, \quad X \in \gm.$$
Moreover, the $\nabla^c$-parallel transports along these geodesics
are given by $\Exp(tX)_*$. The canonical connection has parallel
curvature and torsion. In general, any $G$-invariant tensor is
$\nabla^c$-parallel. So, in particular any geometric tensor, such as
the metric tensor and the Riemannian curvature tensor, is
$\nabla^c$-parallel.

An important class of this spaces are the so-called naturally reductive spaces.
Geometrically, $M$ is naturally reductive if the Riemannian geodesics coincide with
the $\nabla^c$-geodesics. This definition is equivalent to the following algebraic
condition:
$$\langle [X, Y]_\gm, Z \rangle + \langle Y , [X, Z]_\gm \rangle = 0$$
for all $X, Y, Z \in \gm \simeq T_pM$. In fact, this follow from the
Koszul formula and the Killing equation.

For a naturally reductive space we can compute explicitly the
Levi-Civita connection and the canonical connection. In fact,
$$(\nabla_{\widetilde X} \widetilde Y)_p = \mbox{$\frac{1}{2}$} [\widetilde X, \widetilde Y](p)
= - \mbox{$\frac{1}{2}$} [X, Y]_\gm,$$
$$(\nabla_{\widetilde X}^c \widetilde Y)_p = [\widetilde X, \widetilde Y](p) = - [X, Y]_\gm,$$
where $\widetilde W$ is the Killing field on $M$ induced by $W \in
\gm$. In fact, since $\nabla_{\widetilde X} {\widetilde X} = 0$, one
has that $\nabla_{\widetilde X} {\widetilde Y} = -\nabla_{\widetilde
Y} {\widetilde X}$, and so
$$[\widetilde X, \widetilde Y] =
\nabla_{\widetilde X} {\widetilde Y} - \nabla_{\widetilde Y}
{\widetilde X} = 2\nabla_{\widetilde X} {\widetilde Y}.$$ The second
identity is a direct consequence of the formula of the Lie
derivative in terms of the flow.

\smallbreak

A distinguishes class of naturally reductive spaces are the normal
homogeneous spaces. In this case $\gm$ is the orthogonal
com\-ple\-ment of $\gh$ with respect to a bi-invariant metric in
$G$, that is,
$$\gg = \gh \oplus \gh^\bot$$
is the associated reductive decomposition. Then the quotient
projection $G \to M$ is a Riemannian submersion and so, maps
horizontal geodesics of $G$ into geodesics of $M$.


\section{$G$-invariant fields and the canonical connection}\label{G-invariant}

Let $M = G/G_p$ ($p \in M$) be a Riemannian homogeneous space with a
reductive decomposition $\gg = \gh \oplus \gm$. If $g \in G$, $\gg =
\Ad(g)\gh \oplus \Ad(g)\gm$ is also a reductive decomposition, which
has associated the same canonical connection, since $\nabla^c$ is
$G$-invariant. Thus we obtain the geodesics through $q = g \cdot p$,
which are given by $\Exp(tY)g \cdot p$, $Y \in \Ad(g)\gm$, and the
parallel trans\-ports along these geodesics are given by
$\Exp(tY)_*$.

\smallbreak

Next, we prove that the flow of any $G$-invariant fields is
$\nabla^c$-affine, for a normal homogeneous space.

Let $G$ be a Lie group acting on $M$ by $\nabla^c$-affine diffeomorphisms and let
$X$ be a $G$-invariant field on $M$ with local flow $\varphi_t$ (i.e.\ $g_*(X) = X$
for all $g \in G$, which is equivalent to the fact that any element  of $G$ commutes
with $\varphi_t$). Let
$$F_p = \{q \in M: G_p \cdot q = q\}$$
be the set of fixed points of the isotropy at $p$. Then, as it is
well known, $F_p$ is a closed and totally geodesic submanifold of
$M$. Moreover,
\begin{equation}\label{2.1}
\varphi_t(F_p^o) = F_p^o,
\end{equation}
where $F_p^o$ is the connected component of $F_p$ that contains $p$.
In fact, let $q \in F_p$, then $\varphi_t(q) = \varphi_t(G_p \cdot
q) = G_p \cdot \varphi_t(q)$, hence $\varphi_t(F_p) \subset F_p$.
Applying this to $\varphi_{-t}$, we obtain $\varphi_t(F_p^o) =
F_p^o$.

The next step is to observe that if the canonical conection is nice
(the closure of $G$ is compact in this case, since $M$ is compact)
and $G$ acts transitively on $M$, then the isotropy group does not
change along the fixed points, that is
\begin{equation}\label{puntos fijos}
G_p = G_q \quad \text{for all $q \in F_p$.}
\end{equation}
This is easy to prove. In fact, if $q = g \cdot p \in F_p$ then
$G_q = g G_p g^{-1}$, but $G_p = g G_p g^{-1}$.

\begin{nota}\label{2.3}
Let $G$ be a Lie group and $H$ a Lie subgroup of $G$, then the Lie
algebra of $H$ is $\Ad(N(H))$-invariant, where $N(H)$ is the
normalizer of $H$ in $G$. Moreover, if $G$ admits a bi-invariant
metric then, the orthogonal subspace to the Lie algebra of $H$ is
also $\Ad(N(H))$-invariant.
\end{nota}

%

\begin{teo}\label{2.4}
Let $M = G/G_p$ be a compact normal homogeneous space and let
$\nabla^c$ be the canonical connection associated to the reductive
decomposition $\gg = \gg_p \oplus \gg_p^\bot$. Then the flow of any
$G$-invariant field is $\nabla^c$-affine.
\end{teo}

\begin{proof}
Let $X$ be a $G$-invariant field with associated flow $\varphi_t$.
We will prove that $\varphi_t$ maps geodesics into geodesics and
$\nabla^c$-parallel fields along these geodesics into
$\nabla^c$-parallel fields along the image geodesics. Note that from
this fact we obtain that $\varphi_t$ is $\nabla^c$-affine.

By the homogeneity of $M$ it suffices to prove this for geodesics
starting at $p$. These geodesics are given by $\Exp(sY)\cdot p$, $Y
\in \gg_p^\bot$. Since $X$ is $G$-invariant, $\varphi_t$ commutes
with $G$ and hence
$$\varphi_t(\Exp(sY)\cdot p) = \Exp(sY) \cdot \varphi_t(p).$$
Now, by equalities \ref{2.1} and \ref{puntos fijos}, $\varphi_t(p)
\in F_p$ and $G_{\varphi_t(p)} = G_p$. Let $g \in G$ be such that
$\varphi_t(p) = g \cdot p$. Observe that $g$ must lie in the
normalizer of $G_p$ in $G$. In fact,
$$(g^{-1} G_p g) \cdot p = (g^{-1} G_p) \cdot (g \cdot p) = (g^{-1} \cdot (G_{g \cdot p} \cdot (g \cdot p)) = p.$$
Hence $g^{-1}G_p g \subset G_p$ and thus $g \in N(G_p)$.

But, since $\Ad(g)\gg_p^\bot = \gg_p^\bot$ (by Remark \ref{2.3}),
$\varphi_t(\Exp(s Y) \cdot p)$ is a geodesic through $g \cdot p$.
Finally, the $\nabla^c$-parallel transport along $\Exp(sY) \cdot p$
are given by $\Exp(sY)_*$, which is mapped to $\Exp(sY)_* \circ
d\varphi_t$. Hence $\nabla^c$-parallel fields along geodesics are
mapped into $\nabla^c$-parallel fields.
\end{proof}

\begin{nota}\label{rem2.4}
Note that, with the same proof of Theorem \ref{2.4}, if $\varphi$ is
a diffeomorphism that commutes with $G$, then $\varphi$ is
$\nabla^c$-affine.
\end{nota}


\begin{cor} \label{killingafin}
Let $M = G/H$ be a compact normal homogeneous space and let $X$ be a $G$-invariant
field on $M$, then $X$ is a Killing field.
\end{cor}

The proof Corollary \ref{killingafin} is consequence from Theorem
\ref{2.4} and the following proposition.

\begin{prop}\label{2.5}
Let $M$ be a compact naturally reductive space with associated canonical connection
$\nabla^c$. Then $\Aff_0(\nabla^c) \subset \Iso(M)$.
\end{prop}

\begin{proof}
Recall that $\nabla^c$ and $\nabla$ have both the same geodesics.
Now, if $\varphi$ is $\nabla^c$-affine, $\varphi$ maps (Riemannian)
geodesics into geodesics, hence $\varphi$ is $\nabla$-affine (since
$\nabla$ is torsion free, see \cite[pag.\ 107]{P}). To complete the
proof we need the following well-known lemma, for which we include a
conceptual proof.
\end{proof}

\begin{lema}
Let $M$ be a compact Riemannian manifold and let $X$ be an affine
Killing field on $M$ (i.e.\ the flow associated to $X$ preserves the
Levi-Civita connection). Then $X$ is a Killing field on $M$.
\end{lema}

\begin{proof}
Since $X$ is an affine Killing field, then $X$ is a Jacobi field
along any geo\-de\-sic (since the flow of $X$ map geodesics into
geodesics, due to the fact that the flow of $X$ is given by affine
transformations). Take a unit speed geodesic $\gamma(t)$ and let us
write $X(\gamma(t)) = J^T(t) + J^\bot(t)$ as the sum of two
perpendicular Jacobi fields along $\gamma(t)$, where $J^T(t) =
(at+b)\gamma'(t)$. Since $M$ is compact, any field is bounded, thus
$a = 0$. Hence $J(t) = b\gamma'(t) + J^\bot(t)$. Differentiating
with respect to $t$ both sides of the equation $b = \langle J(t),
\gamma'(t) \rangle$, we obtain that
$$\langle \nabla_{\gamma'} X, \gamma' \rangle = b\langle \nabla_{\gamma'} \gamma', \gamma'\rangle = 0.$$
Since $\gamma(t)$ is arbitrary, $X$ satisfies the Killing equation.
\end{proof}

\begin{nota} \label{compact}
Let $M$ be a compact naturally reductive space. By Proposition
\ref{2.5}, the connected component of the group of affine
transformations of $\nabla^c$ is contained in $\Iso(M)$. Moreover,
as it is not difficult to see, $\Aff_0(\nabla^c)$ is a closed
subgroup of $\Iso(M)$, and thus $\Aff_0(\nabla^c)$ is compact. For
an arbitrary linear connection in a compact manifold, the connected
component of the affine group needs not to be compact (see for
instance \cite{Z}).
\end{nota}



\section{The transvections and the affine group}\label{grupoafin}

Let $M = G/H$ be a compact normal homogeneous space and consider
$\gk$ the Lie algebra of $G$-invariant fields on $M$. In the
previous section we see that $\gk \subset \aff(\nabla^c)$, the Lie
algebra of $\Aff_0(\nabla^c)$. So, $\gk$ determines a connected Lie
subgroup $K \subset \Aff_0(\nabla^c)$. If $X$ is a $G$-invariant
field, then the flow associated to $X$ commutes with $G$. Hence $G$
and $K$ commute and we have
$$G \subset G \cdot K \subset \Aff_0(\nabla^c).$$
The purpose of this section is to prove that $G \cdot K =
\Aff_0(\nabla^c)$.

The key of the proof is to show that $G$ is a normal subgroup of
$\Aff_0(\nabla^c)$ and choose a complementary ideal of $\gg =
\Lie(G)$, which is contained in $\gk$. Note that this depends
strongly on the fact that $\Aff_0(\nabla^c)$ is a compact Lie group
(Remark \ref{compact}). In fact, we will prove that $G$ coincides
with the group of transvections of the canonical connection.

The canonical transvection group consists of all $\nabla^c$-affine
transformations that preserve the $\nabla^c$-holonomy subbundles of
the orthogonal frame bundle.

\begin{nota}
Let $M = G/H$ be a homogeneous space, where $H$ is the isotropy at
$p \in M$. It is well-known that the $\nabla^c$-parallel transport
along any curve is realized by elements of $G$. Then, if $\varphi$
is a transvection, $\varphi \in G$. Hence
$$\Tr(\nabla^c) \subset G.$$
Moreover, $\Tr(\nabla^c)$ is a connected and normal subgroup of
$\Aff_0(\nabla^c)$ (and so of $G$), which is transitive on $M$. The
Lie algebra of $\Tr(\nabla^c)$, as it is well-known, is given by
$$\tr(\nabla^c) = [\gm, \gm] + \gm$$
(not a direct sum, in general). The transvection group needs not to
be a closed subgroup of $G$. If we present $M =
\Tr(\nabla^c)/\Tr(\nabla^c)_p$, then the original canonical
connection coincides with the canonical connection associated to the
reductive decomposition $\tr(\nabla^c) = [\gm, \gm]_\gh + \gm$.
\end{nota}

\begin{prop}\label{G = Tr}
Let $M = G/H$ be a compact normal homogeneous space with associated
canonical connection $\nabla^c$. Then $G = \Tr(\nabla^c)$. In
particular, $G$ is a normal subgroup of $\Aff_0(\nabla^c)$.
\end{prop}

\begin{proof}
Recall that $\Tr(\nabla^c)$ is a normal subgroup of $G$. Let $\ga$
be the orthogonal complement of $\tr(\nabla^c)$, that is $\gg =
\tr(\nabla^c) \oplus \ga$ (orthogonal sum). Since $\tr(\nabla^c) =
[\gm,\gm] + \gm$, then $\ga$ is orthogonal to $\gm$, and hence $\ga
\subset \gh$. So, $\ga$ is invariant by any isotropy group, since
any two isotropy groups are conjugates. Since $G$ acts effectively
on $M$, we conclude that $\ga = 0$.
\end{proof}

\begin{proof}[Proof of Theorem \ref{main}]
We will make use of $G$-invariant fields (whose associated flow is
$\nabla^c$-affine, by Theorem \ref{2.4}). Now, since $G$ is a normal
subgroup, then the Lie algebra of $G$, $\gg$ is an ideal of
$\aff(\nabla^c)$. Let $\gg'$ be a complementary ideal, i.e.\
$$\aff(\nabla^c) = \gg \oplus \gg'$$
(for example, choosing $\gg'$ to be the orthogonal complement with respect to a
bi-invariant metric in $\Aff_0(\nabla^c)$, which is compact). If $X \in \gg'$, then
$[X, \gg] = 0$. So, $X$ belongs to $\gk$, the Lie algebra of $G$-invariant fields.
Then $\Aff_0(\nabla^c) \subset G \cdot K$, which implies the equality. To complete
the proof of Theorem \ref{main} see the proof of Corollary \ref{main2} below.
\end{proof}

\begin{proof}[Proof of Corollary \ref{main2}]
From Theorems \ref{main} and \ref{fullisometry}, one has that
$$\Iso_0(M) = \Aff_0(\nabla^c) = G \cdot K.$$
Now, since $G_0$ is abelian and $G$ commutes with $K$, $G_0 \subset K$. Moreover, if
$X$ is a Killing field induced by $G_1$ which coincides with a Killing field induced
by $K$, one has $[X, \gg_1] = 0$, where $\gg_1 = \Lie(G_1)$, since $G_1$ commutes
with $K$. Hence $X = 0$, since $G_1$ is semisimple.
\end{proof}


\begin{nota} \label{nrcase} {\it The naturally reductive case.}
Let $M$ be a compact naturally reductive space with associated
canonical connection $\nabla^c$, and let $p \in M$. One has that
$\Aff_0(\nabla^c)$ is a compact subgroup of $\Iso(M)$ (by Remark
\ref{compact}). Take the presentation $M =
\Tr(\nabla^c)/\Tr(\nabla^c)_p$. In this case one cannot prove that
the flows of $\Tr(\nabla^c)$-invariant fields are $\nabla^c$-affine,
since the reductive complement is not necessarily
$\Ad(N(\Tr(\nabla^c)_p))$-invariant (see Remark \ref{2.3}). However,
with the same arguments of this section, to calculate the affine
group one can supplement $\tr(\nabla^c)$ with an ideal of
$\aff(\nabla^c)$. In fact, if $\tilde \gk$ is the Lie algebra
generated by the $\Tr(\nabla^c)$-invariant fields whose flows are
$\nabla^c$-affine, and if $\widetilde K$ is the connected Lie
subgroup of $\Aff(\nabla^c)$ associated to $\tilde \gk$, then
$$\Aff_0(\nabla^c) = \Tr(\nabla^c) \cdot \widetilde K,$$
and $\Tr(\nabla^c)$ commutes with $\widetilde K$.
\end{nota}

\begin{nota} Let $M = G/ H$ be (compact) naturally reductive
with an associated canonical connection $\nabla ^c$. It is a
well-known fact that the affine group $\Aff(\nabla^c)$ is given by
the diffeomorphisms of $M$ which map $\nabla^c$-geodesics into
$\nabla^c$-geodesics and preserve the torsion tensor. If $M$ is
simply connected there is a standard way of enlarging the group
$G$. More precisely, any linear isometry $\ell: T_pM \to T_q M$,
with $\ell(R^c_p) = R^c_q$ and $\ell(T^c_p) = T^c_q$ extends to an
isometry of $M$ (since the canonical connection has
$\nabla^c$-parallel curvature $R^c$ and torsion $T^c$). In certain
cases the standard extension is trivial. For example for $S^7 =
\Spin(7)/G_2$ (strongly isotropy irreducible presentation, see
\cite{Wo}).
\end{nota}

\subsection{The fixed points of the isotropy}\label{isofixed}

We close this section with a comment on the fixed points of the full isotropy. Let
$M = G/G_p$ be a (compact) normal homogeneous space and let $\nabla^c$ be the
associated canonical connection. Let $F_p$ be the connected component by $p$ of the
set of fixed points of $G_p$ (recall that the isotropy does not change along the set
of fixed points, by \ref{puntos fijos}). This induces a foliation $\mathcal{F}$ of
$M$ with totally geodesics leaves $F_p$, since $\nabla^c$-affine transformations are
isometries of $M$. Note that $F_p$ is a Lie group with a bi-invariant metric. In
fact, consider the subgroup $H$ of $G$ such that leaves $F_p$ invariant. Since $G$
is transitive and preserves the foliation $\mathcal{F}$, one must have that $H$ is
simply transitive on $F_p$. So, $F_p$ is naturally identified with $H$ endowed with
a bi-invariant metric. In fact, the right invariant fields on $H \simeq F_p$ are the
restrictions to $F_p$ of the Killing fields induced by $H$. Moreover, the left
invariant fields on $H$ correspond to the restrictions to $F_p$ of the $G$-invariant
fields on $M$, which are Killing fields on $M$, since the metric is normal
homogeneous (see Theorem \ref{2.4}).

\begin{proof}[Proof of Theorem \ref{torus}]
In order to apply Corollary \ref{main2} we can assume $M$ to be irreducible
(eventually, by passing to the universal cover of $M$; see \cite[Remark 6.5]{OR}).
We keep the notation of the above paragraph. Let $S_p$ be the connected component by
$p$ of the set of fixed points of $\Iso_0(M)_p$. Then $S_p$ is a totally geodesic
submanifold of $M$ which is contained in $F_p$. Moreover, it is not hard to see that
$S_p$ coincides (locally) with the Euclidean de Rham factor of the symmetric space
$F_p$ (see the above paragraph). This implies that $S_p$ is a flat torus.
\end{proof}

\section{Application to homogeneous fibrations}\label{applications}

In this section we will show that the holonomy group of a homogeneous fibration (i.e.\ the fibers are given by
the orbits of a compact Lie group) may be regarded as a subgroup of the affine transformations of the fibers,
with respect to some canonical connection. In particular, this implies that this holonomy group must be a Lie
group \cite{GW}.

General results on submersions $\pi: M \to B$ are due to Ehresmann and Hermann. Namely, {\it if any fiber is
connected and compact, then $\pi$ is a fiber bundle} \cite{E, R}. Moreover, {\it if the submersion is
Riemannian, $M$ is complete, and $M, B$ are both connected, then $\pi$ is a fiber bundle} \cite{H, R}.

Let $\pi: M \to B$ be a $G$-homogeneous (metric) fibration, where $\pi$ is a
Riemannian submersion, $M$ is complete and $G$ is a compact subgroup of the
isometries of $M$ (and all the $G$-orbits are principal). One can decompose
orthogonally $TM = \cv \oplus \ch$, where $\cv = \ker d\pi$ is the vertical
distribution and $\ch$ is the so-called horizontal distribution. One can lift
horizontally curves in $B$, and this lift is unique for any arbitrary initial
condition in the fiber. Denote by $M_b := \pi^{-1}(\{b\}) = G \cdot b$ the fiber at
$b$. The holonomy group $\Phi_b$ of $\pi$ at $b \in B$ is the group of holonomy
diffeomorphisms of $M_b$ induced by horizontal lifts of loops at $b$. Observe that,
in general, $\Phi_b$ is not necessarily a Lie group, a counterexample can be found
in \cite[Remark 9.57]{Be}.


Let us show that the holonomy group $\Phi_b$ is a subgroup of affine
transformations with respect to a canonical connection in the fiber
$M_b$ (which is a normal homogeneous space, since $G$ is compact).
We fix some notation before continuing. Let $\alpha: [0, 1] \to B$,
$\alpha(0) = \alpha(1) = b$, be a loop at $b$. For $p \in M_b$, let
$\tilde\alpha_p$ be the unique horizontal lift of $\alpha$ with
$\tilde\alpha_p(0) = p$. The holonomy diffeomorphism $\varphi \in
\Phi_b$, induced by $\alpha$, is given by $\varphi(p) =
\tilde\alpha_p(1)$. Since the horizontal distribution is
$G$-invariant, $\varphi$ commutes with $G$, i.e.\ $\varphi(g \cdot
p) = g \cdot \varphi(p)$.

Each fiber $M_b$ can be regarded as a normal homogeneous space. In
fact, let $H$ be the isotropy subgroup at $p$ and let $\nabla^c$ be
the canonical connection associated to the reductive decomposition
$\gg = \gh \oplus \gh^\bot$. Let $\varphi \in \Phi_b$. Since
$\varphi$ commutes with $G$, we can apply Remark \ref{rem2.4} to
obtain that $\varphi$ is $\nabla^c$-affine. Hence $\Phi_b \subset
\Aff(\nabla^c)$. Moreover, any element of $\Lie(\Phi_b)$ is a
$G$-invariant field.

\begin{cor}[\cite{GW}]\label{hol}
Let $\pi: M \to B$ be a $G$-homogeneous fibration, where $G$ is a compact Lie group,
and let $b \in B$. Then $\Phi_b$ is a Lie group.
\end{cor}

\begin{proof}
Keep the above notation. From the previous comments we have that
$(\Phi_b)_0$ is an abstract subgroup of $\Aff_0(\nabla^c)$. Hence,
from a result of Goto \cite{G} (see also Yamabe \cite{Y}),
$\Phi_b$ is a Lie group, since $\Aff(\nabla^c)$ is a Lie group
(see \cite{KN}).
\end{proof}

\begin{nota}
By Remark \ref{compact}, $\Aff(\nabla^c)$ is a compact Lie group.
Hence the closure of the holonomy group of a homogeneous fibration
$\Phi_b$ is compact in the affine group $\Aff(\nabla^c)$. Moreover,
by Proposition \ref{2.5}, $\Phi_b$ is a Lie subgroup of $\Iso(M_b)$,
where $M_b = G \cdot b$ is endowed with a normal homogeneous metric.
\end{nota}

\begin{nota} Let $\pi: M \to B$ be a $G$-homogeneous fibration,
and let $b \in B$. The fiber $M_b$ at $b$ is a (normal) homogeneous
space, namely $M_b = G/H$, where $H$ is the isotropy group at a
point $p \in M_b$. Assume that $p$ is the only fixed point of $H$ in
$M_b$. Then the holonomy group $\Phi_b$ at $b$ is trivial. In fact,
let $\alpha$ be a loop at $b$ and let $\tilde \alpha$ be the
horizontal lift of $\alpha$ with $\tilde \alpha(0) = p$. Since
$\tilde \alpha(0)$ is fixed by $H$, and $H$ maps horizontal curves
into horizontal curves, we have that $\tilde \alpha(t)$ is fixed by
the isotropy group, for all $t$. In particular, $\tilde\alpha(1) =
p$. Observe that $p$ is arbitrary, since the isotropy groups at
different points are conjugated.
\end{nota}

\end{document}